\documentclass[11pt]{amsart}

\usepackage{hyperref}

\usepackage{graphicx, enumerate, url}
\usepackage{amssymb,mathtools}
\usepackage{algorithm}
\usepackage{algpseudocode}

\numberwithin{equation}{section}
\numberwithin{algorithm}{section}

\theoremstyle{plain}
\newtheorem{theorem}{Theorem}[section]

\newtheorem{lemma}[theorem]{Lemma}

\theoremstyle{definition}

\theoremstyle{remark}

\DeclareMathOperator{\tr}{Tr}

\DeclareMathOperator{\sym}{sym}
\DeclareMathOperator{\psym}{psym}

\newcommand{\N}{\mathbb{N}}
\newcommand{\C}{\mathbb{C}}

\newcommand{\R}{\mathbb{R}}

\newcommand{\x}{\mathbf{x}}

\newcommand{\y}{\mathbf{y}}
\newcommand{\z}{\mathbf{z}}

\newcommand{\cA}{\mathcal{A}}
\newcommand{\cB}{\mathcal{B}}

\newcommand{\cE}{\mathcal{E}}
\newcommand{\cF}{\mathcal{F}}
\newcommand{\cG}{\mathcal{G}}
\newcommand{\cH}{\mathcal{H}}

\newcommand{\cS}{\mathcal{S}}
\newcommand{\cT}{\mathcal{T}}

\newcommand{\ba}{\mathbf{a}}
\newcommand{\bb}{\mathbf{b}}
\newcommand{\bc}{\mathbf{c}}

\newcommand{\bF}{\mathbf{F}}
\newcommand{\bG}{\mathbf{G}}

\newcommand{\bu}{\mathbf{u}}

\newcommand{\bx}{\mathbf{x}}

\newcommand{\rP}{\mathrm{P}}

\newcommand{\rS}{\mathrm{S}}

\newcommand{\0}{\mathbf{0}}

\newcommand{\trans}{^\top}

\begin{document}
\title
{Upper bounds for the spectral norm of symmetric tensors}
\author{Shmuel Friedland}
\date{April 19, 2021}
\begin{abstract}  
The maximum of the absolute value of a real homogeneous polynomial of degree  $d\ge 3$ on the unit sphere corresponds to the spectral norm of the induced real $d$-symmetric tensor $\cS$.
We give two sequences of upper bounds on the spectral norm of $\cS$, which are stated in terms of certain roots of the Hilbert-Schmidt norms of corresponding iterates.  We show that these sequences are converging to a limit, which is the minimal value of these upper bounds.  Some generalizations to iterates of homogeneous polynomial maps are discussed.

\end{abstract}
\maketitle

\maketitle

\section{Introduction} \label{sec:intro}
Let $\R^n$ be the space of real column vectors with $n$-coordinates, and denote by 
$\otimes_{j=1}^d \R^{n_j}$ the space of $d$-mode tensors.  A tensor $\cT\in \otimes_{j=1}^d \R^{n_j}$ has coordinates $t_{i_1,\ldots,i_d}, i_1\in[n_1],\ldots,i_d\in[n_d]$, where $[n]=\{1,\ldots,n\}$.  The space $\otimes_{j=1}^d \R^{n_j}$
has an inner product and the induced Hilbert-Schmidt norm
\begin{equation}\label{definprodHSnorm}
\begin{aligned}
\langle \cS,\cT\rangle:=\sum_{i_j\in[n_j],j\in [d]} s_{i_1,\ldots,i_d}t_{i_1,\ldots,i_d}, \quad
\cS=[s_{i_1,\ldots,i_d}], \cT=[t_{i_1,\ldots,i_d}],\\
\|\cT\|_{HS}=\sqrt{\langle \cT,\cT\rangle}=\sqrt{\sum_{i_j\in[n_j],j\in [d]} |t_{i_1,\ldots,i_d}|^2}.
\end{aligned}
\end{equation}
Recall that that the spectral norm of $\cT$ \cite{FL18} is defined 
\begin{equation}\label{specnormdef}
\begin{aligned}
\|\cT\|_{\sigma}=\\
\max\{|\langle \cT,\otimes_{j\in]d]}\x_j\rangle|, \x_j\in\R^n, j\in[d],
\|\x_1\|_2=\cdots=\|\x_d\|_2=1\}.
\end{aligned}
\end{equation}
Here $\|\x\|_2$ is the Euclidean norm on $\R^n$.  Clearly
\begin{equation}\label{basnrmineq}
\|\cT\|_\sigma\le \|\cT\|_{HS}.
\end{equation}

The space of equidimensional  $d$-mode tensors, $n_1=\cdots=n_d=n$, is denoted by $\otimes^d\R^n$.  Denote by $\rS^d\R^n\subset \otimes^d\R^n$ the subspace of $d$-symmetric tensors.  (A tensor $\cS$ is called symmetric if the value of the entries  $s_{i_1,\ldots,i_d}$ don't change if we permute the indices.)  Banach's theorem \cite{Ban38} yields
\begin{equation}\label{Banthm}
\|\cS\|_{\sigma}=\max\{|\langle \cS,\otimes^d\x\rangle|, \x\in\R^n,\|\x\|_2=1\}, \quad \cS\in\rS^d\R^n.
\end{equation}
It is known that computing the spectral norm of $3$-tensors for $d\ge 3$ is NP-hard \cite{DL08},
and computing the spectral norm of symmetric $4$-tensor is NP-hard \cite{FW20}.

There is an isomorphism between the space of symmetric tensors $\rS^d\R^n$ and the space of homogeneous polynomials $f$ of degree $d$ in $n$-real variables, denoted as $\rP(d,n,\R)$.  Namely $f(\x)=\langle S,\otimes^d\x\rangle$.  Define 
$\|f\|_{HS}:=\|\cS\|_{HS}$ and $\|f\|_\sigma:=\|\cS\|_\sigma$.  It is known that $\|f\|_\sigma$ is the maximum of $|f(\x)|$ on the unit sphere $\|\x\|_2=1$ \cite{FW20}.
Hence the inequality \eqref{basnrmineq} boils down to 
\begin{equation}\label{basfineq}
\|f\|_\sigma\le \|f\|_{HS}
\end{equation}

Let us consider the case of symmetric matrices $S=[s_{i,j}]\in \rS^2\R^n$.
Then $S$ has $n$-real eigenvalues $\lambda_1(S)\ge \cdots\ge \lambda_n(S)$.
Recall that 

\noindent
$\|S\|_\sigma=\max(|\lambda_1(S)|,|\lambda_n(S)|)$.  The sequence $\|S^k\|_{HS}, k\in\N$ is a submultiplicative sequence, and
\begin{equation}\label{upbdsymmat}
\begin{aligned}
\|S^k\|_{HS}^{1/k}=\big(\tr S^{2k})^{1/(2k)}=\big(\sum_{i=1}^n \lambda^{2k}_i(S))^{1/(2k)}, \quad k\in\N\\ 
\textrm{ is a decreasing sequence converging to }\|S\|_{\sigma}.
\end{aligned}
\end{equation}

The following problem was raised by Harald Andres Helfgott on Mathoverflow \cite{Hel21}:
Is it possible to improve the trivial upper bound  \eqref{basnrmineq} for $d(\ge 3)$-mode symmetric tensors, which is equivalent to  \eqref{basfineq}?  Ideally one would like to have a similar result to \eqref{upbdsymmat}. 

In this paper we give two sequences of upper bounds on the spectral norm of $\cS$, which are stated in terms of certain roots of the Hilbert-Schmidt norms of corresponding iterates.  We show that these sequences are converging to a limit, which is the minimal value of these upper bounds.  Some generalizations to iterate of homogeneous polynomial maps are discussed.

We now summarize briefly the contents of this paper.  In Section \ref{sec:prodin}
we discuss some product inequality of norms of symmetrized tensors, which gives rise to the product inequality of the Hilbert-Schmidt norm of homogeneous polynomials $\|f\|_{}$.  Theorem \ref{mtheorema} shows that the sequence $\|f^k\|_{HS}^{1/k}, k\in\N$ is bounded below by $\|f\|_{\sigma}$, and is convergent to $\rho_1(f)$.  Moreover $\|f^k\|_{HS}^{1/k}\ge \|f^{2k}\|_{HS}^{1/(2k)}$.  As in the quadratic case, if $f$ is orthogonally diagonalizable then $\rho_1(f)=\|f\|_{\sigma}$.

In Section \ref{sec:polmap} we mainly discuss polynomial maps $\bF:\R^n\to\R^n$,
where each coordinat of $\bF$ is a homogeneous polynomial of degree $d$.  The linear space of these polynomial maps is isomorphic to the tensor subspace $\R^n\otimes \rS^p\R^n\subset \otimes^{p+1}\R^n$.  Then $\|\bF\|_{HS}$ and $\|\bF\|_\sigma$ correspond to $\|\cF\|_{HS}$ and $\|\cF\|_{\sigma}$ respectively.
Here $\|\bF\|_\sigma$ is the maximum of $\|\bF(\x)\|_2$ on the unit sphere.  Clearly $\|\bF\|_\sigma\le \|\bF\|_{HS}$.
The $k$-th iterate $\bF^{\circ k}$ of $\bF$ corresponds to the $k$-th composition tensor $\cF^{\circ k}\in\R^n\otimes^{p^k}\R^n$.   Theorem \ref{mtheoremb} shows that the two sequences  $\{\|\bF^{\circ k}\|_\sigma^{(p-1)/(p^k-1)}\}$ and $\{\|\bF^{\circ k}\|_{HS}^{(p-1)/(p^k-1)}\}$ converge to $\rho_3(\cF)$ and $\rho_2(\cF)$ respectively,
and $\rho_3(\cF)\le \rho_2(\cF)$.  In the case $\cF$ is symmetric, equivalently $\bF=(1/(p+1))\nabla f$, the first sequence is constant and $\rho_3(\cF)=\|f\|_\sigma$.
Then 
\begin{equation*}
\|f\|_\sigma \le \|\bF^{\circ (2k)}\|_{HS}^{(p-1)/(p^{2k}-1)}\le  \|\bF^{\circ k}\|_{HS}^{(p-1)/(p^{k}-1)}
\end{equation*}
This is the second improvement of  \eqref{basfineq}.  If $f$ is orthogonally diagonalizable then $\|f\|_{\sigma}=\rho_2(\cF)$.

In Section \ref{sec:simpubsn} we discuss simple upper bound on $\|\cT\|_\sigma$ related to the mathoverflow discussions \cite{Hel21}.  In Section \ref{sec:remcom} we briefly point out how to generalize our results to complex valued tensors and homogeneous polynomials.  We also raise two open problems on the nature of $\rho_1(f)$ and $\rho_2(\cF)$.

\section{Product inequalities}\label{sec:prodin}
A basic inequality in operator theory is the submultiplicative inequality $\|QP\|_{ac}\le \|Q\|_{bc} \|P\|_{ab}^q$, where
$P:\R^{n}\to \R^{m}, Q:\R^m\to \R^l$, and $P$ and $Q$ are bounded homogeneous operators: 
\begin{eqnarray*}
P(t\x)=t^p P(\x), \, \x\in\R^n, \quad Q(t\y)=t^qQ(\y),\,\y\in\R^m, \quad t,q>0.
\end{eqnarray*}
Here $\|\x\|_a,\|\y\|_b,\|\z\|_c$ are norms on $\R^n,\R^m,\R^l$ respectively, and $\|P\|_{ab}$ is the induced operator norm $\|P\|_{ab}=\max\{\|P\x\|_b,\|\x\|_a=1\}$.

Recall that for $P\in\R^{m\times n}, Q\in\R^{l\times m}$ one has the submulitplicative inequality $\|QP\|_{HS}\le \|Q\|_{HS}\|P\|_{HS}$ \cite[Sec. 5.6]{HJ}, even though it does not seem to follow from the inequality $\|QP\|_{ac}\le \|Q\|_{bc} \|P\|_{ab}$.

We now discuss some generalizations of these inequalities to tensors.
For general tensors $\cA=[a_{i_1,\ldots,i_p}]\in \otimes_{k=1}^p\R^{m_k},\cB=[b_{j_1,\ldots,j_q}]\in\otimes_{l=1}^q\R^{n_l}$ one can define the tensor (Kronecker) product $\cA\otimes\cB=[a_{i_1,\ldots,i_p}b_{j_1,\ldots,j_q}]\in(\otimes_{k=1}^p\R^{m_k})\otimes(\otimes_{l=1}^q \R^{n_l})$.
From the definitions of the Hilbert-Schmidt and the spectral norm 
it is straightforward to deduce
\begin{equation}\label{tenprodeq}
\begin{aligned}
\|\cA\otimes \cB\|_{HS}=\|\cA\|_{HS}\|\cB\|_{HS},\\
\|\cA\otimes \cB\|_{\sigma}=\|\cA\|_{\sigma} \| \cB\|_\sigma.
\end{aligned}
\end{equation}

We first discuss symmetric tensors.
Recall that that $p$-symmetric tensors $\rS^p\R^n$ are closely related to real homogeneous polynomials of degree $p$ in $n$ variables, denoted as $\rP(p,n,\R)$.
Assume that $\cF\in \rS^p\R^n$.  Then
\begin{equation}\label{defpolfx}
 \begin{aligned}
f(\x)=\sum_{ j_k+1\in [p+1],k\in[n], j_1+\cdots +j_n=p} \frac{p!}{j_1!\cdots j_n!} \phi_{j_1,\ldots,j_n} x_1^{j_1}\cdots x_n^{j_n}=\\
\langle \cF,\otimes^p \x\rangle=
\sum_{i_1,\ldots,i_p\in[n]}f_{i_1,\ldots,i_p}x_{i_1}\cdots x_{i_p}\in\rP(p,n,\R).
\end{aligned}
\end{equation}
Note that $\phi_{j_1,\ldots,j_n}=f_{i_1,\ldots,i_p}$, where $j_k$ is the number of times $k$ appears in the multi sequence $\{i_1,\ldots,i_p\}$ \cite{FW20}.
Vice versa, given a polynomial $f(\x)\in \rP(p,n,\R)$ there exists a unique $\cF\in\rS^p\R^n$ such that $f(\x)=\langle \cF,\otimes^p\x\rangle$.
Define $\|f\|_{HS}:=\|\cF\|_{HS}$ and  $\|f\|_\sigma:=\|\cF\|_{\sigma}$. 
Recall \cite{FW20}:
\begin{equation}\label{defnormf(x)}
\begin{aligned}
\|f\|_{HS}=\sqrt{\sum_{ j_k+1\in [p+1],k\in[n], j_1+\cdots +j_n=p} \frac{p!}{j_1!\cdots j_n!}|\phi_{j_1,\ldots,j_p}|^2}=\|\cF\|_{HS},\\
\|f\|_\sigma=\max\{|f(\x)|, \|\x\|_2=1\}=\max\{\frac{|f(\x)|}{\|\x\|_2^p}\}=\|\cF\|_\sigma.
\end{aligned}
\end{equation}

Given two symmetric tensors $\cF\in\rS^p\R^n$ and $\cG\in\rS^q\R^n$ we let $f(\x)=\langle \cF,\otimes^p\x\rangle$ and $g(\x)=\langle \cG,\otimes^q\x\rangle$.  Clearly $fg\in\rP(p+q,n,\R)$.  Hence 
one can define $\cF\underset{\rm sym}{\otimes}\cG=\cH\in \rS^{p+q}\R^n$, such that
$f(x)g(x)=\langle \cH, \otimes^{p+q}\x\rangle$.  

Let $\sym:\otimes^d\R^n\to \rS^d\R^n$ be the symmetrization map.  That is 
for $\cT=[t_{i_1,\ldots,i_d}]\in\otimes^d\R^n$ define $\cS=[s_{i_1,\ldots,i_d}]=\sym(\cT)$ by
\begin{equation*}
s_{i_1,\ldots,i_d}=\frac{1}{d!}\sum_{\sigma\in\Sigma_d}t_{i_{\sigma(1)},\ldots,i_{\sigma(d)}}.
\end{equation*}
Here $\Sigma_d$ is the set of bijections $\sigma:[d]\to [d]$.  
\begin{lemma}\label{symprodinlem}  Assume that $\cT\in\otimes^d\R^n,\cF\in\rS^p\R^n$ and $\cG\in\rS^q\R^n$.  Then
\begin{enumerate}[(a)]
\item  $\|\sym(\cT)\|_{HS}\le \|\cT\|_{HS}$.  Equality holds if and only if  $\sym(\cT)=\cT$.
\item $\langle \cT,\otimes^d\x\rangle=\langle \sym(\cT),\otimes^d\x\rangle$.
\item $\|\sym(\cT)\|_{\sigma}\le \|\cT\|_{\sigma}$.  Equality holds if and only if there exists $\x\in\R^n, \|\x\|_2=1$ such that $\|\cT\|_\sigma=|\langle \cT, \otimes^d\x\rangle|$.
\item $\cF\underset{\rm sym}{\otimes}\cG=\sym(\cF\otimes \cG)$.
\item $\|\cF\underset{\rm sym}{\otimes}\cG\|_{HS}\le \|\cF\|_{HS}\|\cG\|_{HS}$.  Equality holds if and only if either $\cF$ or $\cG$ is a zero tensor, or $\cF$ and $\cG$ are rank one symmetric tensors of the form $a\otimes^p\bc, b\otimes^q\bc$ for some $\bc\in \R^n\setminus\{\0\}, a,b\in\R\setminus\{0\}$.
\item Let $f(\x)=\langle \cF,\otimes^p\x\rangle$ and $g(\x)=\langle\cG,\otimes^q\x\rangle$.
Then 
\begin{enumerate}[(i)]
\item $|f(\bc)|\le \|f\|_{HS}\|\bc\|_2^p$.  For $\bc\ne \0$ equality holds if and only if there exists $a\in\R$ such that $f(\x)=a(\bc^\top\x)^p$ for all $\x\in\R^n$.
\item $\|f g\|_{HS}\le \|f\|_{HS}\|g\|_{HS}$.  Equality holds if and only if either $f$ or $g$ is a zero polynomial, or there exists $\bc\in\R^n\setminus\{\0\},a,b\in\R\setminus\{0\}$ such that $f(\x)=a (\bc^\top \x)^p$ and $g(\x)=b (\bc^\top \x)^q$.
\item $\|fg\|_\sigma\le \|f\|_\sigma \|g\|_\sigma$.  Equality holds if and only if there exists $\x\in\R^n, \|\x\|_2=1$, such that $\|f\|_\sigma=|f(\x)|$ and $\|g\|_\sigma=|g(\x)|$.
\item For a positive integer $k$ equality $\|f^k\|_\sigma=\|f\|_{\sigma}^k$ holds.
\end{enumerate}
\end{enumerate}
\end{lemma}
\begin{proof} (a)  follows from the inequality $k| \big(\sum_{i=1}^k a_k\big)/k|^2\le \sum_{i=1}^k |a_k|^2$.  Furthermore equality holds if and only if $a_1=\cdots=a_k$.

\noindent
(b) follows from the fact that $\otimes^d\x$ is a symmetric tensor.

\noindent
(c) follows from (b), and the characterizations \eqref{specnormdef} and \eqref{Banthm}.  The equality case is straightforward.

\noindent
(d) Recall that $\cF\underset{\rm sym}{\otimes}\cG=\cH$, where $\cH$ a unique symmetric tensor such that $\langle \cH,\otimes^{p+q}\x\rangle=\langle \cF\otimes\cG,\otimes^{p+q}\x\rangle$.  By the definition $\langle \cF\otimes\cG,\otimes^{p+q}\x\rangle=\langle\sym(\cF\otimes \cG), \otimes^{p+q}\x\rangle$.  Hence
$\cF\underset{\rm sym}{\otimes}\cG=\sym(\cF\otimes\cG)$.  

\noindent
(e) Use (d) and the inequality in (a) $\|\sym(\cF\otimes \cG)\|_{HS}\le  \|\cF\otimes \cG\|_{HS}$ to deduce the inequality $\|\cF\underset{\rm sym}{\otimes}\cG\|_{HS}\le \|\cF\otimes\cG\|_{HS}$.  Now use the first equality of \eqref{tenprodeq} to deduce the inequality $\|\cF\underset{\rm sym}{\otimes}\cG\|_{HS}\le \|\cF\|_{HS}\|\cG\|_{HS}$.  

We now discuss the equality case in the above inequality. By (a) equality holds if and only $\cF\otimes\cG=\sym(\cF\otimes\cG)$.  
Clearly if either $\cF$ or $\cG$ is a zero tensor the above equality holds.  Thus we assume that $\cF$ and $\cG$ are nonzero tensors and $\cF\otimes \cG$ is a symmetric tensor.

Suppose first that $p=1$, that is, $\cF=\bc\in\R^n\setminus\{\0\}$.  We now show by induction on $q$ that $\cG=b\otimes^q\bc$.
Let us first consider the well known case where $p=q=1$.  Then $\cF=\bc,\cG=\y\in\R^n$.  Thus $\cF\otimes\cG=\bc\otimes\y$.  Then $\bc\otimes\y$ is a symmetric matrix if and only if $\y=b\bc$.  
Assume next that $p=1$ and $q>1$.  Then $\cF=[c_i], \cG=[g_{j_1,\ldots,j_q}]$.  Assume that $[c_ig_{j_1,\ldots,j_q}]$ is a symmetric tensor in the indices $i,j_1,\ldots,j_q]$.  In particular if we fix the indices $j_2,\ldots,j_q$ we obtain that the matrix
$[c_i g_{j_1,\ldots,j_q}]$ is symmetric in $i,j_1$.  Hence from the case $p=q=1$ we deduce that there exists $b_{j_2,\ldots,j_q}$ such that $b_{j_1,\ldots,j_q}=b_{j_2,\ldots,j_k}c_{j_1}$ for $j_1,\ldots,j_q\in[n]$.  Thus $\cG=\bc\otimes\cB $, where $\cB\in \otimes^{q-1}\R^n$.  As $\cG$ symmetric and $\bc\ne \0$ it follows that $\cB$ is a symmetric tensor.  As $\cG\ne 0$ we deduce that $\cB$ is a nonzero symmetric tensor.  Hence by the induction assumption $\cB=b\otimes^{q-1}\bc$, and $\cG=b\otimes^q\bc$.  

We now consider the case where $p>1$.  Assume that $f_{i_1,\ldots,i_p}\ne 0$.
Fix the indices $i_1,\ldots,i_{p-1}$ and let $c_i=f_{i_1,\ldots,i_{p-1},i}$ for $i\in[n]$.  
Observe next that the $1+q$ mode tensor obtained from $\cF\otimes \cG=[f_{i_1,\ldots,i_p} g_{j_1,\ldots,j_q}]$ by fixing the coordinates $i_1,\ldots,i_{p-1}$ is a nonzero symmetric tensor of the form $\bc\otimes \cG$.  Hence the case $p=1$ yields
that $\cG=b\otimes^q\bc$.  As $\cF\otimes\cG$ is symmetric it follows that $\cF\otimes\cG=\cG\otimes\cF$.  The previous arguments yield that $\cF=a\otimes^p\bc$.

Clearly, if $\cF=a\otimes^p\bc$ and $\cG=b\otimes^q\bc$ then $\cF\otimes \cG$ is a symmetric tensor.

\noindent
(f) (i) follows from the Cauchy-Schwarz inequality.

\noindent
(ii) follows from (e).

\noindent
(iii) and (iv) follow from the second characterization in \eqref{defnormf(x)}.
\end{proof}

Note that part (a) of the above lemma follows from the well known fact that the tensor $\sym(\cT)$ is the orthogonal projection of $\otimes^d\R^n$ on $\rS^d\R^n$.  

We now give a first analog of improvement of the inequality \eqref{basnrmineq}.
\begin{theorem}\label{mtheorema}
Assume that $f(\x)$ is a nonzero real homogeneous polynomial of degree $p\ge 1$ in $n$ variables. 
\begin{enumerate}[(a)]
\item Suppose that $f$ is proportional to a $p$-power of a linear form: $f(\x)=a(\bc^\top\x)^p$ for some $\bc\in\R^{n}\setminus\{\0\}, a\ne 0$.  Then
\begin{equation}\label{mtheorema1}
\|f^k\|_{HS}=\|f^k\|_{\sigma}=|a|^k\|\bc\|_2^{kp}, \quad k\in\N.
\end{equation}
\item Assume that $f$ is not proportional to a $p$-power of a linear form.
Then for $k,l\in\N$
\begin{equation}\label{fineq}
\begin{aligned}
\|f^{k+l}\|_{HS}< \|f^k\|_{HS}\|f^l\|_{HS},\\
\|f^{2k}\|_{HS}< \|f^k\|_{HS}^2,\\
\|f\|_{\sigma}^k< \|f^k\|_{HS}.
\end{aligned}
\end{equation}
 In particular the limit below exists and satisfies the inequality:
\begin{equation}\label{limfkHSnorm}
\rho_1(f):=\lim_{k\to\infty} \|f^k\|_{HS}^{1/k}\ge \|f\|_\sigma.
\end{equation}
Furthermore $\rho_1(f)<\|f^k\|_{HS}^{1/k}$ for each $k\in\N$.
\end{enumerate}
\end{theorem}
\begin{proof}
(a)  Recall that $f(\x)=a(\bc^\top \x)^p$ if and only if $f(\x)=\langle a\otimes^p\bc,\otimes^p\x\rangle$.  Then $f^k(\x)=\langle a^k\otimes^{kp}\bc,\otimes^{kp}\x\rangle$.  Part (f) of Lemma \ref{symprodinlem} yields \eqref{mtheorema1}.

\noindent
(b) Suppose that $f$ is not proportional to a $p$-power of a linear form.
Then $f^k$  is not proportional to a $kp$-power of a linear form for each $ k\in\N$.
Part (f)(ii) of Lemma \ref{symprodinlem} yields that the sequence $\|f^k\|_{HS}$ is strictly submultiplicative.  That yields the first inequality in \eqref{fineq}.
Set $k=l$ to deduce the second inequality in \eqref{fineq}.
The second equality of \eqref{defnormf(x)} and part (f)(i)
yield that $\|f\|_{\sigma}<\|f\|_{HS}$.  Hence $\|f\|_{\sigma}^k=\|f^k\|_{\sigma}<\|f^k\|_{HS}$.  Fekete's subadditive lemma yields the existence of the the limit $\rho_1(f)$.
Clearly $\rho_1(f)\ge \|f\|_{\sigma}$.  The second inequality in \eqref{fineq} shows that the sequence
$\|f_{k2^l}\|_{HS}^{1/(k2^{l})}$ is a strictly decreasing sequence for $l=0,1,\ldots$.
Hence $\rho_1(f)<\|f^k\|_{HS}^{1/k}$ for each $k\in\N$.

\end{proof}

We now show that if $f$ is proportional to a power of an orthogonally diagonalizable $g$ then
 $\rho_1(f)=\|f\|_\sigma$.  We call $g\in \rP(q,n,\R)$ orthogonally diagonalizable if there exists an orthogonal matrix $Q\in \R^{n\times n}$ such that 
\begin{equation}\label{diagg}
g(Q\x)=\sum_{i=1}^n \lambda_ix_i^q, 
\end{equation}
where  $\lambda_1,\ldots,\lambda_n\in\R$.
Note that a real quadratic polynomial is orthogonally diagonalizable.

We first recall the well known lemma:
\begin{lemma}\label{HSocv}  Let $\cT\in\otimes_{j=1}^d \R^{n_j}$.  Assume that $Q_j\in\R^{n_j\times n_j}$ is an orthogonal matrix for $j\in[d]$.  Denote by $\cT'\in\otimes_{j=1}^d \R^{n_j}$ the tensor obtained from $\cT$ by the equality
\begin{equation*}
\langle T',\otimes_{j=1}^d \x_j\rangle =\langle T,\otimes_{j=1}^d (Q_j\x_j)\rangle, \x_j\in\R^{n_j}, \,j\in[d].
\end{equation*}
Then $\|\cT'\|_{HS}=\|\cT\|_{HS}$.
\end{lemma}
\begin{proof}
First observe that it is enough to consider the case where all but one $Q_j$ is the identity matrix.  Next we can assume that $Q_j$ is the identity matrix for $j>1$.
Let $Q_1=[q_{i_{d+1},i_{1}}], i_1,i_{d+1}\in[n_1]$.  Thus
\begin{eqnarray*}
\cT=[t_{i_1,\ldots,i_d}], \quad \cT'=[t'_{i_1,\ldots,i_d}],\\
t'_{i_1,\ldots,i_d}=\sum_{i_{d+1}=1}^{n_1}q_{i_{d+1},i_1}t_{i_{d+1},i_2,\ldots,i_d},\\
\|\cT'\|_{HS}=\sum_{i_j\in[n_j],j\ge 2} \sum_{i_1=1}^{n_1} \big(\sum_{i_{d+1}=1}^{n_1}q_{i_{d+1},i_1}t_{i_{d+1},i_2,\ldots,i_d}\big)^2=\\
\sum_{i_j\in[n_j],j\ge 2} \sum_{i_1=1}^{n_1} t^2_{i_1,\ldots,i_d}=\|\cT\|^2_{HS}.
\end{eqnarray*}
The equality one before last follows from the observation that $Q^\top$ preserves the length of vectors in $\R^{n_1}$.
\end{proof}
Denote by $\rP(n,p,\R_+)\subset \rP(n,d,\R)$ the cone of homogeneous polynomials $f$ whose all coefficients $\phi_{j_1,\ldots,j_n}$ in \eqref{defpolfx} are nonnegative.  Assume that $g\in\rP(n,d,\R)$ and $g$ has the expansion \eqref{defpolfx} with coefficients $\gamma_{j_1,\ldots,j_n}$.  We say that $f\in \rP(n,p,\R_+)$ majorizes $g$, denoted as $f\succeq g$, if we have the inequality $\phi_{j_1,\ldots,j_n}\ge |\gamma_{j_1,\ldots,j_n}|$ for all nonnegative integers $j_1,\ldots,j_n$ such that $j_1+\cdots+j_n=p$.  The following inequality is straightforward:
\begin{equation}\label{majin}
\|f^k\|_{HS}\ge \|g^k\|_{HS} \textrm{ if } f\succeq g.
\end{equation}
\begin{theorem}\label{diagthm}  Assume that $f\in\rP(p,d,\R)$ is proportional to a power of an orthogonally diagonalizable $g$.  Then
\begin{equation}\label{diagthm1}   
\rho_1(f)=\|f\|_\sigma.
\end{equation}
\end{theorem}
\begin{proof}  Assume that $g\in\rP(n,q,\R),q\in\N$.  Let $g_1(\x)=g(Q\x)$ for an orthogonal matrix $Q\in\R^{n\times n}$.  Then $\|g_1^k\|_\sigma=\|g^k\|_\sigma=\|g\|_\sigma^k$.  Lemma \eqref{HSocv} yields that $\|g_1^k\|_{HS}=\|g^k\|_{HS}$.  Assume that $g$ is diagonalizable.  Then without loss of generality we can assume that $g(x)=\sum_{i=1}^n \lambda_i x_i^q$.

Suppose first that $q=1$.  (Note that in this case any linear function is equal to some $g$.)  
Then $\|g\|_\sigma=\|g\|_{HS}$.  As $g(\x)$ induced by a rank one symmetric tensor 
part (f) of Lemma \ref{symprodinlem} yields that 
\begin{equation*}
\|g^k\|_{\sigma}=\|g\|_\sigma^k=\|g^k\|_{HS}=\|g\|_{HS}^k.
\end{equation*}
Hence $\rho_1(g)=\|g\|_\sigma$.  Therefore \eqref{diagthm1}  holds for $f=ag^p$.

 Assume now that $q\ge 2$.   Without loss of generality we can assume that $\max\{|\lambda_i|, i\in[n]\}=1$.  Then 
\begin{eqnarray*}
|g(\x)|\le \sum_{i=1}^n|\x_i|^q=\|\x\|_q^q\le \|\x\|_2^{q}, \quad \|\x\|_q=\big(\sum_{i=1}^n |x_i|^q\big)^{1/q}.
\end{eqnarray*}
The last inequality follows from the well known fact that $\|\x\|_q$ is a decreasing function in $q>0$ for a fixed $\x$.  Hence $\|g\|_\sigma=|g((1,0,\ldots,0)^\top)=1$.
It is left to show that $\rho_1(g)=1$.  In view of \eqref{majin} it is enough to assume that  $g=\sum_{i=1}^n x_i^q$.  We claim that 
\begin{eqnarray*}
g^k=\sum_{j_i+1\in[k+1], i\in[n], j_1+\ldots+j_n=k} \frac{k!}{j_1!\cdots j_n!}x_1^{j_1q}\cdots x_n^{j_nq},\\
n\le \|g^k\|_{HS}^2=\\\sum_{j_i+1\in[k+1], i\in[n], j_1+\ldots+j_n=k} \big(\frac{k!}{j_1!\cdots j_n!}\big)^2\big(\frac{(j_1q)!\cdots (j_nq)!}{(kq)!}\big)\le\\ {n+k-1\choose k-1}={n+k-1\choose n}\le (n+k-1)^n.
\end{eqnarray*}
Indeed, the first equality is the multinomial expansion of $g$.  As $g^k\succeq\sum_{i=1}^n x_i^{kq}$ we deduce that $\|g^k\|_{HS}^2\ge n=\|\sum_{i=1}^n x_i^{kq}\|_{HS}^2$.  To find the exact formula for $\|g^k\|_{HS}$ we need to use the first equality in \eqref{defnormf(x)}.  Let $f=g^k$.  Observe that $f$ has degree $kq$.  Hence nonzero monomial $x^{k_1}\cdots x^{k_n}$ in the expansion of $f$ is of the form $x_1^{j_1q}\cdots x_n^{j_nq}$.  Use  the first equality in \eqref{defpolfx} to deduce that 
\begin{equation*}
\phi_{kj_1,\ldots,kj_n}=\big(\frac{k!}{j_1!\cdots j_n!}\big)\big(\frac{(j_1q)!\cdots (j_nq)!}{(kq)!}\big).
\end{equation*}
Use the second equality in \eqref{defpolfx} to deduce the formula of $\|f\|_{HS}^2$ as given above.   We next observe the inequality
\begin{eqnarray*}
\big(\frac{k!}{j_1!\cdots j_n!}\big)^2\le \big(\frac{k!}{j_1!\cdots j_n!}\big)^q\le \frac{(kq)!}{(j_1q)!\cdots (j_nq)!}{(kq)!}.
\end{eqnarray*}
As $q\ge 2$ and all the multinomial coefficients are integer the first inequality is obvious.  The second inequality is shown as follows.  Let $y_i=x_i^q, i\in[n]$.  Then
$\frac{k!}{j_1!\cdots j_n!}$ is the coefficient of $y_1^{j_1}\cdots y^{j_n}$ in the expansion of $(\sum_{i=1}^n y_i)^k$.  Hence
\begin{equation*}
(\sum_{i=1}^n y_i)^{kq}=((\sum_{i=1}^n y_i)^{k})^{q}\succeq
  \big(\frac{k!}{j_1!\cdots j_n!}y_1^{j_1}\cdots y^{j_n}\big)^q,
\end{equation*}
which proves the second inequality.
The third inequality comes from the following observations.  In the formula f or $\|g^k\|_{HS}^2$ we showed that each summand is at most $1$.  We claim that the number of summands in the expansion of $g^k$ is ${n+k-1\choose k-1}$.  Indeed, the number of monomials of degree $k$ in $n$ variables is the dimension of $\rS^k\R^n$, which is equal to ${n+k-1\choose k-1}$ \cite{FW20}.  Clearly ${n+k-1\choose k-1}={n+k-1\choose n}\le (n+k-1)^n$.  Hence $n^{1/2k}\le \|g^k\|_{HS}^{1/k}\le (n+k-1)^{n/2k}$. Let $k\to\infty$ to deduce the equality $\rho_1(g)=1$.  

Assume now that $f=g^l$.   Use the above inequalities to deduce that $\rho_1(f)=1=\|f\|_{\sigma}$.  Hence $\rho_1(af)=|a|\rho_1(f)=\|af\|_{\sigma}$.
\end{proof}

\section{Polynomial maps}\label{sec:polmap}
Assume that $d\ge 3$ is an integer and $k\in[d-1]$.  Suppose that $\cT=[t_{i_1,\ldots,i_k}]\in \otimes_{j=1}^d \R^{n_j}$ 
 and $\cA=[a_{i_{k+1},\ldots,i_d}]\in\otimes_{j=k+1}^d \R^{n_j}$ be $d-k$ tensor.  Then 
\begin{equation*}
\cT\times \cA=\sum_{i_j\in[n_j], j> k} t_{i_1,\ldots,i_d}a_{i_{k+1},\ldots,i_d}\in\otimes_{j=1}^k \R^{n_j}
\end{equation*}
is the contraction of $(\cT,\cA)$ on the indices of $\cA$. 

Assume that $F\in\R^{n\times n}$ is a symmetric matrix and $f(\x)=\x^\top F\x$.
Then $F\x=\frac{1}{2}\nabla f(\x)$.  Assume that $\cF$ is a $p+1$-symmetric tensor in $n$-real variables and $f(\x)=\langle \cF,\otimes^{p+1}\x\rangle$.
As in \cite{FW20} the analog $F\x$ is
\begin{equation}\label{defFx}
\begin{aligned}
\bF(\x)=(F_1(\x),\ldots,F_n(\x))^\top=\cF\times \otimes^{p}\x,\\
F_i(\x)=\sum_{i_2,\ldots,i_p\in[n]} f_{i,i_2,\ldots,i_p}x_{i_2}\cdots x_{i_p}, \quad i\in[n].
\end{aligned}
\end{equation}
Note that each $F_i(\x)$ is a homogeneous polynomial of degree $p+1$.  It is straightforward to show \cite{FW20}
\begin{equation*}
\bF(\x)=\frac{1}{p+1}\nabla f(\x).
\end{equation*}

We now consider polynomial maps $\bF=(F_1(\x),\ldots,F_m(\x))^\top:\R^n\to \R^m$,
where $F_1,\ldots,F_m\in\rP(p,n,\R)$.
Denote by $\R^m\otimes \rS^{p}\R^n$ the subspace of $\R^m\otimes^{p}\R^n$ of all tensors $\cF=[f_{i_1,\ldots,i_{p+1}}]$ which are symmetric with respect to the last $p$ indices.  Observe that
$\bF(\x)=\cF\times \otimes^{p}\x$.
For $p=1$ we have $\R^m\otimes \rS^{1}\R^n=\R^{m\times n}$.  
  
As in the scalar case $m=1$ there is $1-1$ correspondence between $\bF$ and $\cF$.  Define 
\begin{equation}\label{defHSbF}
\|\bF\|_{HS}:=\sqrt{\sum_{i=1}^m \|F_i\|_{HS}^2}.
\end{equation}
Clearly $\|\bF\|_{HS}=\|\cF\|_{HS}$.

A simple modification of Banach's result \cite{Fri13} yields 
\begin{equation*}
\|\cF\|_\sigma=\max\{|\langle\cF,\y\otimes^{p}\x\rangle|, \y\in\R^m,\x\in\R^n,\|\y\|_2=\|\x\|_2=1\}.
\end{equation*}
Taking in account that $\cF\times\otimes^{p}\x=\bF(\x)$ we deduce an analogous result to \cite[(8)]{FW20} 
\begin{equation}\label{Fnormform}
\|\cF\|_\sigma=\max\{\frac{\|\bF(\x)\|_2}{\|\x\|_2^{p}}, \x\ne \0\}.
\end{equation}

Given a tensor $\cT\in\R^m\otimes^p\R^n$ we can define a partial symmetrization 
map $\psym:\R^m\otimes^p\R^n\to \R^m\otimes\rS^d\R^n$, which is a partial symmetrization of $\cT$ with respect to the last $p$ coordinates of of $\cT$.  That is
$\cF=\psym(\cT)$ if 
\begin{equation*}
f_{i_{1},i_2,\ldots,i_p}=\frac{1}{p!}\sum_{\sigma\in\Sigma_p}t_{i_1,i_{\sigma(1)+1},\ldots,i_{\sigma(p)+1}}.
\end{equation*}
Note that 
\begin{equation*}
\cT\times\otimes^p\x=\psym(\cT)\times\otimes^p\x, \quad \x\in\R^n.
\end{equation*}

Given $\cF\in\R^m\otimes \rS^{p}\R^n,\cG\in \R^l\otimes \rS^{q}\R^m$ we can define a composition (``product'') $\cG\circ \cF\in \R^{l}\otimes \rS^{pq}\R^n$ which corresponds to the composition
\begin{equation*}
(G\circ F)(\x)=G(F(\x))=(\cG\circ\cF)\times \otimes^{pq}\x.
\end{equation*}
That is $\cG\circ\cF=\cH\in \R^l\otimes \rS^{pq}\R^n$ such that 
\begin{equation*}
\cH\times(\otimes^{pq}\x)=\bG(\bF(\x)).
\end{equation*}
Note that this product is associative:
\begin{equation*}
(\cG\circ\cF)\circ\cE=\cG\circ(\cF\circ\cE).
\end{equation*}

Assume that $m=n$, that is, $\bF:\R^n \to \R^n$.  Denote
\begin{equation*}
\cF^{\circ k}:=\underbrace{\cF\circ\cdots\circ\cF}_k, \quad k\in\N.
\end{equation*}
So $\cF^{\circ k}$ corresponds of the $k$-th iteration of the map $\bF$ which is also denoted as $\bF^{\circ k}$.  Then $\bF^{\circ 0}(\x)=\x$, and $\cF^{\circ 0}$ is the identity matrix in $\R^n\otimes\R^n$.

The following lemma generalizes some parts Lemma \ref{symprodinlem}.
\begin{lemma}\label{psymmaj}  Assume that $\cT\in\R^m\otimes (\otimes^p\R^n), \cF\in \R^m\otimes\rS^p\R^n\setminus\{0\}, \cG\in\R^l\otimes\rS^q\R^m\setminus\{0\}$.  
Denote $\bF(\x)=\cF\times\otimes^p\x$ and $\bG(\x)=\cG\times\otimes^q\x$.
Then
\begin{enumerate}[(a)]
\item $\|\psym(\cT)\|_{HS}\le \|\cT\|_{HS}$.  Equality holds if and only if $\psym(\cT)=\cT$.
\item $\|\bG\circ\bF\|_{HS}\le \|\bG\|_{HS}\|\bF\|_{HS}^q$.  Equality holds 
if and only there exists 
\begin{eqnarray*}
\ba=(a_1,\ldots,a_n)^\top\in\R^n, 
\|\ba\|_2=1,\\  \bb=(b_1,\ldots,b_m)^\top,\in\R^m\setminus\{\0\},\\
 \bc=(c_1,\ldots,c_l)^\top,\in\R^l\setminus\{\0\}
\end{eqnarray*}
such that
\begin{equation*}
\begin{aligned}
F_i(\x)=b_i(\ba^\top\x)^p,\quad  i\in[m],\, \x\in\R^n,\\
G_j(\y)=c_j(\bb^\top \y)^q,\quad j\in[l],\, \y\in \R^m.
\end{aligned}
\end{equation*}
\item
\begin{equation}\label{bascompin1}
\|\cG\circ\cF\|_{\sigma}\le \|\cG\|_\sigma\|\cF\|_\sigma^{q}.
\end{equation}
Equality holds if and only if there exists $\x\in\R^n,\|\x\|_2=1$ such that $|\|\bF(\x)\|_2=\|\bF\|_{\sigma}$, and for $\y=\frac{1}{\|\bF\|_{\sigma}}\bF(\x)$ the equality  
$\|\bG(\y)\|_2=\|\bG\|_{\sigma}$ holds.
\end{enumerate}
\end{lemma}
\begin{proof} (a) The proof is analogous to the proof part (a) of Lemma \ref{symprodinlem}.

\noindent
(b) We first consider the case where $m=1$.  Then $\bG(\x)=g(\x)\in \rP(q,\R)$.  
Consider first the case $g_{j_1,\ldots,j_m}(\y)=y_1^{j_1}\cdots y_m^{j_m}$, where $j_1,\ldots,j_m$ are nonnegative integers such that $j_1+\cdots+j_m=q$.  Part (f)(ii) of Lemma \ref{symprodinlem} yields 
\begin{equation}\label{monomcase}
\|g_{j_1,\ldots,j_m}(\bF)\|_{HS}=\|F_1^{j_1}\cdots F_m^{j_m}\|_{HS}\le \|F_1\|_{HS}^{j_1}\cdots\|F_m\|_{HS}^{j_m}.
\end{equation}
We now consider the equality case in the above inequality.  We need to discuss those $F_{j_r}$ such that $j_r\ge 1$ for $r\in[m]$.  If $F_{j_r}$ is a zero polynomial for some $j_r\ge 1$ then all the quantities in \eqref{monomcase} are zero.   So assume that $F_{j_r}\ne 0$ if $j_r\ge 1$.  The the equality case in part (f)(ii) of Lemma \ref{symprodinlem} yields that equality holds if and only if $F_{j_r}=b_{j_r}(\ba^\top\x)^p$ for $\ba\in\R^m, \|\ba\|_2=1, b_{j_r}\ne 0$ if $j_r\ge 1$.

Recall that a general $g\in\rP(q,n,\R)$ is of the form
\begin{equation*}
g(\y)=\sum_{j_1+1,\ldots,j_m+1\in[q+1], \sum_{r=1}^m j_r=q} \frac{q!}{j_1!\cdots j_m!} \gamma_{j_1,\ldots,j_m} g_{j_1,\ldots,j_m}(\y).
\end{equation*}
Then
\begin{eqnarray*}
\|g(\bF))\|_{HS}\le \\
\sum_{j_1+1,\ldots,j_m+1\in[q+1], \sum_{r=1}^m j_r=q} \frac{q!}{j_1!\cdots j_m!} |\gamma_{j_1,\ldots,j_m}| \|g_{j_1,\ldots,j_m}(\bF)\|_{HS}\le\\
\sum_{j_1+1,\ldots,j_m+1\in[q+1], \sum_{r=1}^m j_r=q} \big(\sqrt{\frac{q!}{j_1!\cdots j_m!}}|\gamma_{j_1,\ldots,j_m}|\big)\\
\big(\sqrt{\frac{q!}{j_1!\cdots j_m!}}\|F_1\|^{j_1}\cdots\|F_{j_m}\|_{HS}^{j_m}\big)\le\\
\|g\|_{HS}\big( \sum_{j_1+1,\ldots,j_m+1\in[q+1], \sum_{r=1}^m j_r=q} \frac{q!}{j_1!\cdots j_m!}\|F_1\|^{2j_1}\cdots\|F_{m}\|_{HS}^{2j_m}\big)^{1/2}=\\
\|g\|_{HS}\|\bF\|_{HS}^q.
\end{eqnarray*}
Equality in the last inequality holds if there exists $b\in\R$ such that 
\begin{equation*}
|\gamma_{j_1,\ldots,j_m}|=|b|\|F_1\|^{j_1}\ldots\|F_m\|^{j_m}, j_k+1\in[q], k\in[m],j_1+\cdots+j_m=q.
\end{equation*}
Equality in the first inequality holds if the all polynomials 

\noindent
$ \frac{q!}{j_1!\cdots j_m!} \gamma_{j_1,\ldots,j_m} g_{j_1,\ldots,j_m}(\bF(\y))$ are proportional with a nonnegative constant.  Recall also the equality case in \eqref{monomcase}.
After permuting the coodinates of $F_1,\ldots,F_m$ we can assume that $F_1,\ldots,F_l\ne 0$, and $F_{l+1},\ldots,F_m=0$ for $l\in[m]$.  Hence we deduce that $g$ depends only on $y_1,\ldots,y_l$.  Equality in \eqref{monomcase} yields that
$F_i=b_i(\ba^\top \x)^q, \|\ba\|_2=1$ where $b_i\ne 0$ for $i\in[l]$ and $b_i=0$ for $i>l$.
Hence equality in 
\begin{equation*}
\|g(\bF))\|_{HS}\le \|g\|_{HS}\|\bF\|_{HS}^q
\end{equation*}
holds if and only if $\gamma_{j_1,\ldots,j_m}=cb_1^{j_1}\cdots b_m^{j_m}$ for all allowable $j_1,\ldots,j_m$.  That is, $g(\y)=c(\bb^\top\y)^q$.  This proves the equality case for $l=1$.  

Similar arguments apply for the case $l>1$. 

\noindent
(c)
Clearly
\begin{eqnarray*}
\|(\cG\circ \cF)(\x)\|_2=\|\bG(\bF(\x)\|_2\le \|\cG\|_\sigma \|\bF(\x)\|_2^{q}\le 
\|\cG\|_\sigma  \|\cF\|_\sigma^{p}\|\x\|_2^{pq}.
\end{eqnarray*}
Hence \eqref{bascompin1} holds.  Equality holds if and only if there exists $\x\in\R^n,\|\x\|_2=1$ such that $|\|\bF(\x)\|_2=\|\bF\|_{\sigma}$, and for $\y=\frac{1}{\|\bF\|_{\sigma}}\bF(\x)$ the equality  
$\|\bG(\y)\|_2=\|\bG\|_{\sigma}$ holds.
\end{proof}

We first state well known results for square real matrices, which generalizes \eqref{upbdsymmat}. Recall that the spectral radius of $T\in\C^{n\times n}$, denoted as $\rho(T)$, is the maximum of the modulus of the eigenvalues of $T$.
\begin{lemma}\label{specradmat}  Let $F\in\R^{n\times n}\setminus\{0\}$.  Then 
\begin{enumerate}[(a)]
\item Assume that $F$ is rank-one symmetric matrix $a \ba\ba^\top, \|\ba\|_2=1, a\ne 0$.  Then $\|F^k\|_{HS}=|a|^k, k\in\N$.
\item The sequence $\|F^k\|_{HS}$ is submultiplicative for $k\in\N$.  
Suppose that $F$ is not a rank one symmetric matrix.  
Then $\|F^{k+1}\|_{HS}< 
\|F^k\|_{HS}\|F\|_{HS}$ for $k\in\N$.  Furthermore $\|F^{k+l}\|_{HS}=\|F^k\|_{HS} \|F^l\|_{HS}$ if and only if $\min(k,l)\ge 2$ and $F^{\min(k,l)}$ is a rank one symmetric matrix.
\item The spectral radius of $F$ is equal to $\lim_{k\to\infty}\|F^k\|^{1/k}_{HS}$.  If $F$ is symmetric then $\|F\|_{\sigma}=\rho(F)$.
\end{enumerate}
\end{lemma}
\begin{proof}
(a) is straightforward.

\noindent
(b) The submultiplicativity of the sequence $\|F^k\|_{HS}$ is well known and follows from part (b) of Lemma \ref{psymmaj}.  The inequality $\|F^{k+1}\|_{HS}< 
\|F^k\|_{HS}\|F\|_{HS}$ follows from part (b) of Lemma \ref{psymmaj} and the assumption that $F$ is not a symmetric rank one matrix.  

Assume that $\|F^{k+l}\|_{HS}=\|F^k\|_{HS} \|F^l\|_{HS}$.  Then the inequality $\|F^{k+1}\|_{HS}< \|F^k\|_{HS}\|F\|_{HS}$ ylelds that $\min(k,l)\ge 2$.  Assume that $l\ge k$.  Then  part (b) of Lemma \ref{psymmaj}  yields that $F^k$ is a rank one symmetric matrix $a^k\ba\ba^\top, \|\ba\|_2=1$.  That is, the characteristic polynomial of $F$ is $(z-a)z^{n-1}$ and the minimal polynomial is $(z-a) z^{k'}$, where $2\le k'\le k$ \cite[Chapter 3]{HJ}.  As $l\ge k$ we obtain that $F^l=a^l \ba\ba^\top$.  Hence equality $\|F^{k+l}\|_{HS}=\|F^k\|_{HS} \|F^l\|_{HS}$ holds.

\noindent
The equality  $\rho(F)=\lim_{k\to\infty}\|F^k\|^{1/k}_{HS}$ is the Gelfand formula \cite[Corollary 5.6.14]{HJ}.  Clearly, if $F$ is symmetric then Rayleigh quotient yields $\|F\|_{\sigma}=\rho(F)$.
\end{proof}

We now give an analog of the above lemma, which is a different version of Theorems \ref{mtheorema} and \ref{diagthm}:
\begin{theorem}\label{mtheoremb}  Assume that $\cF\in\R^n\otimes\rS^p\R^n\setminus\{0\}$, where $p\ge 2$ is an integer.  Denote $\bF(\x)=\cF\times\otimes^p\x,\x\in\R^n$.  
Consider the iterations $\bF^{\circ k}$ corresponding to the tensors $\cF^{\circ k}\in \R^{n}\otimes ^{p^k}\R^n$ for $k\in\N$.  

Assume first that 
\begin{equation}\label{specformF}
\cF=a\otimes^{p+1}\ba,\quad\ba\in\R^n, \|\ba\|_2=1,\, a\ne 0.
\end{equation}
Then 
\begin{equation}\label{specformF1}
\begin{aligned}
\cF^{\circ k}=a^{(p^{k}-1)/(p-1)}\otimes^{p^k+1}\ba, \, \|\cF^{\circ k}\|_{HS}=|a|^{(p^{k}-1)/(p-1)},\\
\bF^{\circ k}(\x)=a^{(p^{k}-1)/(p-1)}(\ba^\top \x)^{p^k}\ba, \, \|\bF^{\circ k}\|_{HS}=|a|^{(p^{k}-1)/(p-1)}, \, k\in\N.
\end{aligned}
\end{equation} 

Assume second that $\cF$ is not of the form \eqref{specformF}.  
Then 
\begin{enumerate}[(a)]
\item The sequence $f_k:=\|\cF^{\circ k}\|_{HS}^{(p-1)/(p^k-1)}, k\in\N$ satisfies the following properties:
\begin{equation}\label{Fiter1} 
\begin{aligned}
f_k\le \|\cF\|_{HS}, \quad k\in\N,\\
\|\cF^{\circ (k+l)}\|_{HS}\le  \|\cF^{\circ k}\|_{HS} \|\cF^{\circ l}\|_{HS}^{p^k}, \quad k,l\in \N,\\
\textrm{ the sequence } f_{m2^l} \textrm{ is decreasing for } l=0,1,\ldots, \textrm{ and fixed } m\in\N,\\
\rho_2(\cF):=\lim_{k\to\infty} \|\cF^{\circ k}\|_{HS}^{(p-1)/(p^k-1)}.
\end{aligned}
\end{equation}
Furthermore for $k,l\in\N$  equality in the second inequality holds if and only if the following conditions are satisfied:
\begin{enumerate}[(i)]
\item $\min(k,l)\ge 2$.
\item Assume that $l\ge k\ge 2$.  Then 
\begin{equation}\label{Fcircklform}
\cF^{\circ k}=b\otimes^{p^k+1}\ba, \, \cF^{\circ l}=c\otimes^{p^l+1}\ba,  \,\|\ba\|_2=1,
\,bc^{p^k}=cb^{p^l}.
\end{equation}
\end{enumerate}
\item The sequence $\tilde f_k:=\|\cF^{\circ k}\|_{\sigma}^{(p-1)/(p^k-1)}, k\in\N$ satisfies the following properties:
\begin{equation}\label{Fiter2} 
\begin{aligned}
\tilde f_k\le \|\cF\|_{\sigma}, \quad k\in\N,\\
\|\cF^{\circ (k+l)}\|_{\sigma}\le  \|\cF^{\circ k}\|_{\sigma} \|\cF^{\circ l}\|_{\sigma}^{p^k}, \quad k,l\in \N,\\
\textrm{ the sequence } \tilde f_{m2^l} \textrm{ is decreasing for } l=0,1,\ldots, \textrm{ and fixed } m\in\N,\\
\rho_3(\cF):=\lim_{k\to\infty} \|\cF^{\circ k}\|_{\sigma}^{(p-1)/(p^k-1)},\\ 
\rho_3(\cF)\le \rho_2(\cF).
\end{aligned}
\end{equation}
Furthermore for $k,l\in\N$  equality in the second inequality holds if and only if the following conditions are satisfied:  Either $\cF^{\circ l}=0$ or there exists $\x^\star\in\R^n$ such that 
\begin{eqnarray*}
\|\bF^{\circ l}(\x^\star)\|_2=\|\cF^{\circ l}\|_\sigma, \quad \|\x^\star\|_2=1,\\
\|\bF^{\circ k}(\y)=\|\cF^{\circ k}\|_\sigma, \quad \y=\|\cF^{\circ l}\|_\sigma^{-1} \bF^{\circ l}(\x^\star).
\end{eqnarray*}
\item Assume that $\cF\in\rS^{p+1}\R^n$ and let $f(\x)=\langle \cF,\otimes^{p+1}\x\rangle$.
Then
\begin{equation}\label{Fiter3} 
\|f\|_\sigma=\|\cF\|_{\sigma}\le \rho_2(\cF).
\end{equation}
If $f$ is orthogonally diagonalizable then equality holds in the above inequality.
\end{enumerate}
\end{theorem}
\begin{proof}  Observe by induction that $\cF^{\circ k}\in \R^{n}\otimes ^{p^k}\R^n$ for $k\in\N$.
Assume that \eqref{specformF} holds.  Then $\bF(\x)=a(\ba^\top \x)^p\ba$.
We prove by induction on $k\in\N$ that $\bF^{\circ k}(\x)=a^{(p^{k}-1)(p-1)}(\ba^\top \x)^{p^k}\ba$.  Clearly for $k=1$ the above equality holds.  Assume that the equality holds for $k=l$ and let $k=l+1$.  Then 
\begin{eqnarray*}
\bF^{\circ{l+1}}(\x)=\bF^{\circ{l}}(\bF(\x))=a^{(p^{l}-1)/(p-1)}(\ba^\top F(\x))^{p^l}\ba=\\
a^{(p^{l}-1)/(p-1)+p^l}(\ba^\top\x)^{p^{l+1}}\ba=a^{(p^{l+1}-1)/(p-1)}(\ba^\top\x)^{p^{l+1}}\ba.
\end{eqnarray*}
This proves the first formula on the second line of \eqref{specformF1}.  The first formula on the first line of \eqref{specformF1} follows straightforward.  Now deduce the second formula on the first line of \eqref{specformF1}, and the last formula of \eqref{specformF1}.

We now assume that $\cF$ is not of the form \eqref{specformF}.  

\noindent
(a) Use induction and the inequality in part (b) of Lemma \ref{psymmaj}  to show the inequality 
\begin{eqnarray*}
\|\cF^{\circ(k+1)}\|_{HS}=\|\cF^{\circ k}\circ \cF\|_{HS} \le 
\|\cF^{\circ k}\|_{HS} \|\cF\|_{HS}^{p^k}\le\\ 
\|\cF\|_{HS}^{(p^k-1)/(p-1)+p^k}= \|\cF\|_{HS}^{(p^{k+1}-1)/(p-1)}.
\end{eqnarray*}
This proves the first inequality in \eqref{Fiter1}.  The second inequality in \eqref{Fiter1}
follows straightforward from part (b) of Lemma \ref{psymmaj}.

Observe next that 
\begin{eqnarray*}
f_{m2^{l+1}}^{(p^{m2^{l+1}}-1)/(p-1)}=\|\cF^{\circ (m2^{l+1})}\|_{HS}=\|\cF^{\circ (m2^l)}\circ \cF^{\circ(m2^l)}\|_{HS}\le\\ 
\|\cF^{\circ (m2^l})\|_{HS} \|\cF^{\circ (m2^l)}\|_{HS}^{p^{m2^l}}=f_{m2^l}^{(p^{m2^l}+1)(p^{m2^l}-1)/(p-1)}=f_{m2^{l}}^{(p^{m2^{l+1}}-1)/(p-1)}.
\end{eqnarray*}
Hence the sequence $f_{m2^l}$ is decreasing. 

It is left to show that the sequence $f_k, k\in\N$ converges.  
Let
\begin{eqnarray*}
0\le \alpha:=\liminf_{k\to\infty} f_k\le\beta:= \limsup_{k\to\infty} f_k\le f_1
\end{eqnarray*}
We need to show that $\alpha\ge \beta$.  Assume for simplicity of the exposition that $\alpha>0$.  Since $p\ge 2$
\begin{eqnarray*}
\alpha':=\frac{\log\alpha}{p-1}=\liminf_{k\to\infty}\frac{\log\|\cF^{\circ k}\|_{HS}}{p^k-1}=
\liminf_{k\to\infty}\frac{\log\|\cF^{\circ k}\|_{HS}}{p^k},\\
\beta':=\frac{\log\alpha}{p-1}=\limsup_{k\to\infty}\frac{\log\|\cF^{\circ k}\|_{HS}}{p^k-1}=
\limsup_{k\to\infty}\frac{\log\|\cF^{\circ k}\|_{HS}}{p^k}=\\
p^{-m} \limsup_{k\to\infty}\frac{\log\|\cF^{\circ {k+m}}\|_{HS}}{p^k} \textrm{ for a fixed } m\in\N.
\end{eqnarray*}
Fix $\varepsilon >0$ and assume that $l\gg 1$ such that $p^{-l}\log\|\cF^{\circ l}\|_{HS}\le \alpha' +\varepsilon$, and $p^{-l}\beta\le \varepsilon$.  Use the second inequality in \eqref{Fiter1}  to deduce
\begin{eqnarray*}
p^{-(k+l)}\log\|\cF^{\circ (k+l)}\|_{HS}\le p^{-(k+l)}\log\|\cF^{\circ k}\|_{HS} +p^{-l}\log\|\cF^{\circ l}\|_{HS}.
\end{eqnarray*}
Take $\limsup$ on $k$ to deduce
\begin{eqnarray*}
\beta'\le p^{-l}\beta' +\alpha'+\varepsilon\le \alpha'+2\varepsilon.
\end{eqnarray*}
As $\varepsilon>0$ was arbitrary we deduce that $\alpha'\ge \beta'$.  Hence $\alpha'=\beta'$ and $\alpha=\beta$.  Similar arguments show the equality if $ \alpha=\beta$ if $\alpha=0$.  This proves the last equality of \eqref{Fiter1}.

We now discuss the equality in the second inequality of \eqref{Fiter1}.
Using the equality case in part (b) of Lemma \ref{psymmaj} we obtain
\begin{eqnarray*}
\cF^{\circ k}=\bc\otimes^{p^k} \bb, \quad \cF^{\circ l}=\bb\otimes^{p^l}\ba.   
\end{eqnarray*}
Observe that each coordinate of $\bF^{\circ k}(\bF^{\circ l}(\x))$ and $\bF^{\circ l}(\bF^{\circ k}(\x))$ are proportional to $(\ba^\top \x)^{p^{k+l}}$ and $(\bb^\top \x)^{p^{k+l}}$ respectively.  As $\bF^{\circ k}(\bF^{\circ l}(\x))=\bF^{\circ l}(\bF^{\circ k}(\x))$ it follows that  $\bb$ is proportional to $\ba$, and $\bc$ is proportional to $\bb$.  Hence \eqref{Fcircklform} holds.
Vice versa if \eqref{Fcircklform} holds then equality in the second inequality of \eqref{Fiter1}.

It is left to show that $\min(k,l)\ge 2$.  
Suppose to the contrary that $\min(k,l)=1$.  So $\cF=a\otimes^{p+1}\ba$ contrary to our assumptions. 

\noindent
(b)  Use \eqref{Fnormform} and the proof of the first three inequalities in \eqref{Fiter1}  to deduce the first three inequalities in \eqref{Fiter2}.  The proof of the existence of the limit $\rho_2(\cF)$
in \eqref{Fiter1} yields the existence  of the limit  $\rho_3(\cF)$ in \eqref{Fiter2}.  The inequality
$\rho_3(\cF)\le \rho_2(\cF)$ follows from the inequality \eqref{basnrmineq}.

The equality case in in the second inequality of \eqref{Fiter2} is straightforward.

\noindent
(c)  Assume that $\cF\in\rS^{p+1}\R^n$ and let $f(\x)=\langle \cF,\otimes^{p+1}\x\rangle$.  Let $\bF(\x)=\cF\times \otimes^p\x=\frac{1}{p+1}\nabla f(\x)$.  Thus $f(\x)=\sum_{i=1}^n x_i F_i(\x)$, which is Euler's identity.  The Cauchy-Schwarz inequality yields that $|f(\x)|\le \|\x\|_2\|\|\bF(\x)\|_2$.  Hence $\|f\|_\sigma\le \|\cF\|_{\sigma}$, see \eqref{Fnormform}.
The second equality of \eqref{defnormf(x)} yields there exists $\x^\star$ such that \cite{FW20}
\begin{equation*}
\bF(\x^\star)=\pm \|\cF\|_\sigma \x^\star, \quad \|\x^\star\|_2=1.
\end{equation*}
The equality case in the second inequality of \eqref{Fiter2} yields that $\|\cF^k\|_{\sigma}=\|\cF\|_{\sigma}^{(p^k-1)/(p-1)}$.   Thus $\rho_3(\cF)=\|\cF\|_{\sigma}$.  Then inequality $\rho_3(\cF)\le \rho_2(\cF)$ in \eqref{Fiter2} yields the inequaity $\|f\|_\sigma=\|\cF\|_{\sigma}\le \rho_2(\cF)$.

We now discuss the case where $f$ is orthogonally diagonalizable.
Assume first that $f(\x)=\sum_{i=1}\lambda_i x_i^{p+1}$, where $p\ge 2$.  We can assume without loss of generality that $\lambda_1=1$ and $|\lambda_i|\le 1$ for $i\ge 2$.  As in the proof of Theorem \ref{diagthm} we have that $\|f\|_{\sigma}=1$.  Assume that $\cF\in\rS^{p+1}\R^n$ satisfies $f(\x)=\langle\cF, \otimes^{p+1}\x\rangle$. As in the proof of Theorem \ref{diagthm} to show that $1=\rho_2(\cF)$ it is enough to consider the case $f(\x)=\sum_{i=1}^n x_i^{p+1}$.  Then $\bF(\x)=(x_1^p,\ldots,x_n^p)^\top$.  Thus
$\bF^{\circ k}(\bx)=\sum_{i=1}^n x_i^{p^k}$.  Hence $\|\cF^{\circ k}\|_{HS}=\sqrt{n}$ and $\rho_2(\cF)=1$. 

Assume now that $g(\x)=f(Q\x)$, where $f(\x)=\sum_{i=1}\lambda_i x_i^{p+1}$.
Then $\bG(\x)=Q^\top \bF(Q\x)=Q^{-1}\bF(Q\x)$.  Hence 
$$\bG^{\circ k}(\x)=Q^{-1}\bF^{\circ k}(Q\x)=Q^\top \bF^{\circ k}(Q\x).$$
Then $\|\cG^{\circ k}\|_{HS}=\|\cF^{\circ k}\|_{HS}=\sqrt{n}$ and $\|\cG\|_{\sigma}=\rho_2(\cG)$.
\end{proof}

It is very plausible that $\|f\|_{\sigma}=\rho_2(\cF)$ if $f$ is proportional to a power of an orthogonally diagonalizable $g$.  However the computations are much more involved, so we left out  the treatment of this case.

 \section{A simple upper bound for the spectral norm}\label{sec:simpubsn}
 For $\cT\in \otimes_{j=1}^d \R^{n_j}$ let 
\begin{equation*}
\cT(\x_3,\ldots,\x_d):=\cT\times\otimes_{j=3}^d \x_j\in \R^{n_1\times n_2}, \quad \x_j\in\R^{n_j}, j=3,\ldots,d,
\end{equation*}
be an $n_1\times n_2$ matrix, whose entries are multilinear function on $\R^{n_3},\ldots,\R^{n_d}$.  Note that 
\begin{equation}\label{deftau}
\tau_{\cT}(\x_3,\ldots,\x_d):=\tr \cT(\x_3,\ldots,\x_d)^\top \cT(\x_3,\ldots,\x_d), \,\x_j\in\R^{n_j}, j\ge 3,
\end{equation}
is  sum-of-squares polynomial of degree $2(d-2)$, which is quadratic in each variable $\x_3,\ldots,\x_d$.

Let $L=\{ l_1,\ldots,l_r\}\subset [d]$, whose cardinality $r$ is at least two. 
We say that $\cT$ is $L$-symmetric  if: $n_{l_1}=\cdots=n_{l_r}$, and the values of $t_{i_1,\ldots,i_d}$ do not change if we permute the inidices $i_{l_1},\ldots,i_{l_r}$.  
Note that if $\cT$ is $\{1,2\}$ symmetric then  $\cT(\x_3,\ldots,\x_d)$ is a symmetric matrix of order $n_1$.  

\begin{lemma}\label{simpub}  Let $d\ge 3$ and $\cT\in\otimes^d_{j=1} \R^{n_j}$.
Then 
\begin{enumerate}[(a)]
\item
\begin{eqnarray*}
\|\cT\|_\sigma=\max\{\|\cT(\x_3,\ldots,\x_d)\|_\sigma,  \|\x_3\|_2=\cdots=\|\x_d\|=1\}\le\\
\max\{\sqrt{\tau_\cT(\x_3,\ldots,\x_d)}, \|\x_3\|_2=\cdots=\|\x_d\|=1\}.
\end{eqnarray*}
\item Assume that $d=3$.  Then $\gamma_{\cT}(\x_3)$ is a positive semidefinite quadratic form induced by $T(\cT)\in\rS^2\R^{n_3}$.  Let $\lambda(\cT)$ be the maximum eigenvalue of $T(\cT)$.  Then 
\begin{equation*}
\|\cT\|_\sigma\le \sqrt{\lambda(\cT)}\le \|\cT\|_{HS}.
\end{equation*}
\item Assume that $\cT\in\rS^d\R^n$.  Let $\tilde\tau_\cT(\x)=\tau_{\cT}(\underbrace{\x,\ldots,\x}_{d-2})$.
Then
\begin{equation*}
\|\cT\|_\sigma^2\le \max\{\tilde \tau_{\cT}(\x), \x\in\R^n, \|\x\|_2=1\}.
\end{equation*}
\end{enumerate}
\end{lemma}
\begin{proof}(a)
Clearly 
\begin{eqnarray*}
\|\cT\|_\sigma=\max\{\max\{|\langle\cT,\otimes_{j=1}^d \x_j\rangle|, \|\x_1\|_2=\|\x_2\|_2=1\},\|\x_3\|=\cdots=\x_d\|_2=1\}\\
=\max\{\|\cT(\x_3,\ldots,\x_d)\|_\sigma, \|\x_3\|=\cdots=\x_d\|_2=1\}.
\end{eqnarray*}
This shows the equality case in (a).  Recall that 
\begin{eqnarray*}
\|\cT(\x_3,\ldots,\x_d)\|_\sigma^2\le \tr \cT(\x_3,\ldots,\x_d)^\top \cT(\x_3,\ldots,\x_d)= 
\tau_{\cT}(\x_3,\ldots,\x_d).
\end{eqnarray*}
This proves the inequality part of (a).

\noindent
(b)  Recall that $\|\cT(\x_3)\|_\sigma^2\le \tau_{\cT}(\x_3)$. By definition $\tau_{\cT}(\x_3)$ is a sum-of-squares of linear forms.  Hence $\tau_{\cT}(\x_3)=\x^\top_3 T(\cT)\x_3$.
The maximum of this quadratic form on the sphere $\|\x_3\|_2=1$ is the maximum eigenvalue  $\lambda(T(\cT))$.  Hence (a) yields that $\|\cT\|_\sigma\le \sqrt{\lambda(T(\cT))}$.  

Assume that $\x=(x_1,\ldots,x_{n_3})^\top$ and $\|\x_3\|_2=1$.  Then
\begin{eqnarray*}
\tr \cT^\top(\x_3) \cT(\x_3)=\sum_{i=1,j=1}^{n_1,n_2}|\sum_{k=1}^{n_3} t_{i,j,k}x_k|^2\le\\
\sum_{i=1,j=1}^{n_1,n_2}\big(\sum_{k=1}^{n_3} |t_{i,j,k}|^2\big)\big(\sum_{k=1}^{n_3}|x_k|^2\big)=\sum_{i=1,j=1}^{n_1,n_2}\sum_{k=1}^{n_3} |t_{i,j,k}|^2=\|\cT\|_{HS}^2.
\end{eqnarray*}
This proves the last inequality in (b).

\noindent
(c) Recall Banach's theorem.  Observe that 
\begin{eqnarray*}
|\langle \cT,\otimes^d\x\rangle\le \|\cT(\underbrace{\x,\ldots,\x}_{d-2})\le \sqrt{\hat \tau(\x)}.
\end{eqnarray*}
The last inequality of (a) yields the inequality (c).
\end{proof}

Observe that for $d\ge 4$ an upper bound on a symmetric $\|\cT\|_\sigma^2$, which can be characterized by \cite{Ban38}, is $\|\cS\|_\sigma$ for some symmetric $\cS\in\rS^{2(d-2)}$.

We remark that Will Sawin outlines an  approach to estimate from above $\|f\|_\sigma$ for a real cubic form in \cite{Hel21}, which seems to be different from the upper bound (b) of Lemma \ref{simpub}.
\section{Additional remarks and open problems}\label{sec:remcom}
Denote by $\rP(d,n,\C)$ the vector space of all homogeneous polynomials of degree $d$ in $n$ variables over the complex numbers $\C$.  Let 
\begin{equation*}
\langle \cS,\cT\rangle:=\sum_{i_j\in[n_j], j\in[d]} s_{i_1,\ldots,i_d}\overline{t_{i_1,\ldots,i_d}}
\end{equation*}
be the standard inner product on $\otimes_{j=1}^d \C^{n_j}$.
Then
\begin{eqnarray*}
\|f\|_{\sigma,\C}=\max\{|f(\x)|,\, \|\x\|_2=1, \, \x\in\C^n\}, \quad f\in\rP(n,d,\C),\\
\|\cT\|_{\sigma,\C}=\max\{|\langle \cT,\otimes_{j=1}^d \x_j\rangle|,\, \x_j\in\C^{n_j},\|\x_j\|_2=1, j\in[d]\},\quad \cT\in\otimes_{j=1}^d \C^{n_j}.
\end{eqnarray*}
Recall that Banach's theorem \cite{Ban38} is valid over the complex numbers:
\begin{equation*}
\|\cS\|_{\sigma,\C}=\max\{|\langle \cS,\otimes^d\x\rangle|,\, \x\in\C^n,\, \|\x\|_2=1\}, \quad \cS\in\rS^d\C^n.
\end{equation*}

We claim the results of Sections \ref{sec:prodin} and \ref{sec:polmap} extends to the complex case, provided we do the following modifications:
\begin{itemize}
\item
Replace $\langle \cF,\otimes_{j=1}^d \x_j\rangle$ with $\langle \cF,\otimes_{j=1}^d \overline{\x_j}\rangle$.
\item
In Theorems \ref{diagthm} and \ref{mtheoremb} replace the notion of orthogonally diagonalizable with unitary diagonalizable.  That is, we assume that $Q$ is a unitary matrix in \eqref{diagg} and $\lambda_i$ are complex numbers.
\item
In Lemma \ref{HSocv} one assumes that $Q$ is unitary.
\end{itemize}
(Recall the  Autonne-Takagi  factorization theorem \cite[Corollary 4.4.4, part (c)]{HJ},
which claims that a complex quadratic form is unitary diagonalizable.)

Assume that $d\ge 3$ and $\cS\in\cS^d\R^n$.  Clearly $\|\cS\|_\sigma \le \|\cS\|_{\sigma,\C}$.   It is well known that a strict inequality can hold
 \cite{FL18}.  Let $f$ be the homogeneous polynomial of degree $d$ induced by $\cS$.  Then$\|f\|_{\sigma}\le \|f\|_{\sigma,\C}$ and strict inequality may hold.  Consider the quantities $\rho_1(f)$ and $\rho_2(\cS)$ as defined in \eqref{limfkHSnorm} and \eqref{Fiter1}.   Then the complex versions of Theorems  \ref{mtheorema} and  \ref{mtheoremb} yields the inequalities
\begin{equation}\label{comlub}
\|f\|_{\sigma,\C}\le \min(\rho_1(f),\rho_2(\cS)).
\end{equation}

Let $\|\x\|_r=\big(\sum_{i=1}^n |x_i|^r\big)^{1/r},r\in[1,\infty]$ be the $r$-H\"older norm.
Assume that $d\in\N$ and $f\in\rP(n,d,\R^n)$.  One can define
\begin{equation}\label{defsigrnrm}
\|f\|_{\sigma,r}=\max\{|f(\x)|, \, \|\x\|_r\le 1,\x\in\R^n\}, \quad r\in[1,\infty]
\end{equation}
As $\|\x\|_r$ is a decreasing function in $r\ge 1$ for a fixed $\x$ we deduce
\begin{equation}\label{fnrminc}
\|f\|_{\sigma,r_1}\le \|f\|_{\sigma,r_2} \textrm{ if } 1\le r_1 \le r_2\le \infty.
\end{equation}

As in \cite{LS47,Qi05} it is natural to consider the $d$-spectral norm of 
$f\in\rP(n,d,\R)$.  If $d$ is even,
it is straightforward to show that the maximum $\x$ satisfies the nonlinear eigenvalue equation
\begin{equation}\label{eigeq}
\cT\times \otimes^{d-1}\x=\lambda \x^{\circ (d-1)}, \quad \x^{\circ (d-1)}=(x_1^{d-1},\ldots,x_n^{d-1})^\top, \, \|\x\|_d=1.
\end{equation}
Here $\cT\in \rS^{d}\R^n$ and $f(\x)=\langle \cT,\otimes^d\x\rangle$.  
The above equation makes sense for $\cT\in \R^n\times \rS^{d-1}\R^n$,  $\lambda\in\C$ and $\x\in\C^n$.  The scalar $\lambda$ and the vector $\x$  are called the eigenvalue and the eigenvector of $\cT$.
The maximum absolute value of the real and complex eigenvalues is called the spectral radius of $\cT$, and
denoted by $\rho(\cT)$.   Thus for $d$ even we have the inequality $\|f\|_{\sigma,d}\le \rho(\cT)$.

 Assume that  $\cT=[t_{i_1,\ldots,i_d}]\in \R^n\times \rS^{d-1}\R^n$ and let $|\cT|=[|t_{i_1,\ldots,i_d}|]\in \R^n\times \rS^{d-1}\R^n$.  Then $\rho(|\cT|)$ has the Collatz-Wielandt characterization \cite[(3.10)]{FG20}
 \begin{equation}\label{infmaxcharrho}
 \rho(|\cT|)=\inf_{\x=(x_1,\ldots,x_n)\trans>\0} \max_{i\in [n]}\frac{\big(|\cT|\times \otimes^{d-1}\x\big)_i}{x_i^{d-1}}.
 \end{equation}
 Hence $\rho(\cT)\le \rho(|\cT|)$.  In particular $\||f|\|_{\sigma,d}=\rho(|\cT|)$, where $\cT$ is the symmetric tensor induced by $f$.  Note that characterization \eqref{infmaxcharrho} yields easy upper bounds for $\rho(|\cT|)$.  

Assume in addition that $\cT=|\cT|$.  Suppose furthermore that $\cT$ is weakly irreducible \cite{FGH13}.  Then there is a unique positive eigenvector $\bu$ that satisfies \eqref{eigeq}, and the corresponding eigenvalue is $\rho(\cT)$.  Corollary 5.1 in \cite{FGH13} claims that for any $\y>\0$ the iterations of the map 
\begin{equation*}
\widetilde{\bF}(\y)=\big(\sum_{i=1}^n (F_i(\y))^{1/(d-1)}\big)^{-1}\big((F_1(\y))^{1/(d-1)},\ldots, (F_n(\y))^{1/(d-1)}\big)^\top>\0
\end{equation*}
converge to a positive mulitple of $\bu$.

 \bigskip
We now state the following open problems:
\begin{enumerate}[A.]
\item  Assume that $\cS\in \rS^d\C^n$ and $f(\x)=\langle \cS,\otimes^d\bar\x\rangle$.  Is there a relation between between $\rho_1(f)$ and $\rho_2(\cS)$?
\item Is there a dynamics meaning of $\rho_2(\cF)$ when studying the iterations of $\bF$ in the complex projective space $\mathbb{P}^{n-1}$?
\end{enumerate}
\bigskip
{\em Acknowledgement}: I thank Lek-Heng Lim for pointing out the mathoverflow discussions \cite{Hel21}, and Harald Andres Helfgott for very useful remarks.  The author was partially supported by Simons collaboration grant for mathematicians.

\bibliographystyle{plain}

\end{document}